\documentclass[fleqn,suppldata]{interact}

\usepackage{epstopdf}
\usepackage{subfigure}
\usepackage{amsmath,amssymb,amsfonts}
\usepackage{enumitem}
\usepackage{caption}
\usepackage[numbers,sort&compress,merge]{natbib}
\bibpunct[, ]{[}{]}{,}{n}{,}{,}
\theoremstyle{plain}
\newtheorem{theorem}{Theorem}[section]
\newtheorem{lemma}[theorem]{Lemma}
\newtheorem{corollary}[theorem]{Corollary}
\newtheorem{proposition}[theorem]{Proposition}

\theoremstyle{definition}

\newtheorem{example}[theorem]{Example}

\theoremstyle{remark}

\newcommand{\Complex}{\mathbb{C}}

\newcommand{\la}{{\lambda}}

\begin{document}

\articletype{RESEARCH ARTICLE}

\title{Improved Inclusion-Exclusion Eigenvalue Localization Sets for Matrices and Matrix Polynomials} 

\author{%
	\name{Ljiljana Cvetkovi\'c\textsuperscript{a}\thanks{Corresponding author. Email: lila@dmi.uns.ac.rs},
		Christina Michailidou\textsuperscript{b},
		Irena Prodanovi\'c\textsuperscript{c}}
	\affil{\textsuperscript{a}Faculty of Sciences, University of Novi Sad,
		Trg D. Obradovi\'ca 4, 21000 Novi Sad, Serbia}	
	\affil{\textsuperscript{b}Department of Mathematics, School of Applied Mathematical and Physical Sciences,
		National Technical University of Athens,
		Zografou Campus, 15780 Athens, Greece}
	\affil{\textsuperscript{c}Faculty of Technical Sciences, University of Novi Sad,
		Trg D. Obradovi\'ca 6, 21000 Novi Sad, Serbia}
}

\maketitle

\begin{abstract}
	Recent developments have shown that eigenvalue localization regions can often be significantly sharpened by identifying and excluding subsets that are guaranteed to contain no eigenvalues, leading to the concept of inclusion-exclusion localization sets.
	In this paper, we introduce a new family of inclusion-exclusion eigenvalue localization sets based on the $S-$strict diagonal dominance framework and the corresponding CKV localization regions. For an arbitrary nonempty subset  $S$ of the index set, we construct novel exclusion regions and derive the associated  $S-$SDD inclusion-exclusion localization sets. We prove that every eigenvalue of a complex matrix belongs to the proposed regions and define a refined localization set obtained by intersecting the corresponding inclusion-exclusion regions over all nonempty subsets $S$. The proposed framework unifies and extends several existing results. In particular, we show that the classical Ger\v sgorin inclusion-exclusion localization set and the Dashnic-Zusmanovich inclusion-exclusion localization set arise naturally as special cases of our construction. Consequently, the new localization set provides, in general, a sharper eigenvalue inclusion region than both of these previously known localization regions. The theory is further extended to matrix polynomials, yielding corresponding inclusion-exclusion localization results for polynomial eigenvalue problems. Several numerical examples illustrate the effectiveness of the proposed approach and demonstrate the improvement achieved over existing localization sets.
\end{abstract}

\begin{keywords}
Eigenvalue localization;  $S-$strict diagonal dominance; inclusion-exclusion set; matrix polynomials
\end{keywords}

\section{Introduction}

Eigenvalue localization sets play a central role in matrix analysis and numerical linear algebra. Besides their intrinsic theoretical interest, localization regions provide valuable information on the distribution of eigenvalues and constitute an important tool in stability analysis, iterative methods, control theory, and many other applications.

The classical starting point is the Ger\v sgorin theorem \cite{Ger}, which localizes the spectrum of a matrix inside a union of disks. Since its appearance, numerous refinements and extensions have been developed. Among the most influential are the Brauer Cassini ovals \cite{Bra}, Ostrowski-type regions \cite{Ost}, and various localization sets based on generalized diagonal dominance \cite{Vargabook}.

A different direction of research emerged from the observation that many localization sets contain substantial regions that cannot contain eigenvalues. This led to the development of exclusion regions and inclusion-exclusion localization techniques. Early contributions in this direction can be traced back to Schneider \cite{Sch}, who introduced exclusion regions for eigenvalues. Later, Melman \cite{Mel} obtained the Ger\v sgorin inclusion-exclusion set by removing suitable exclusion disks from the classical Ger\v sgorin region. More recently, Zhao and She \cite{ZS} developed exclusion regions for the Dashnic-Zusmanovich localization set, producing substantially sharper eigenvalue estimates.

A common feature of these approaches is that they start from a given inclusion set and identify subregions that are guaranteed to be free of eigenvalues. The resulting inclusion-exclusion localization sets often provide considerably tighter spectral estimates than the original localization regions.

On the other hand, Cvetkovi\'c, Kosti\'c and Varga \cite{CKV} introduced a family of localization sets based on arbitrary subsets $S$ of the index set. These sets unify several classical localization techniques and lead to the CKV localization set obtained by intersecting all corresponding $S-$localization regions. The flexibility of the subset $S$ provides a natural framework for constructing localization regions based on $S-$strict diagonal dominance and for unifying several previously known eigenvalue localization techniques.

The purpose of this paper is to combine the ideas of $S-$localization and inclusion-exclusion localization. We introduce a new family of exclusion regions associated with arbitrary subsets $S$ and construct the corresponding $S-$SDD inclusion-exclusion localization sets. The proposed framework contains, as particular limiting cases, both the Ger\v sgorin inclusion-exclusion localization set and the Dashnic-Zusmanovich inclusion-exclusion localization set.

More importantly, the new localization set is obtained as the intersection of inclusion-exclusion regions corresponding to all nonempty subsets $S$ of the index set. This leads to a localization region that is, in general, sharper than the previously known Ger\v sgorin and Dashnic-Zusmanovich inclusion-exclusion sets.

The notions of Ger\v sgorin, Brauer and Dashnic-Zusmanovich inclusion sets have been generalized to the case of matrix polynomials in \cite{MP18}, while the corresponding exclusion sets (for matrix polynomials) have been introduced and studied in \cite{MPP}. It is worth mentioning that, in the case of matrix polynomials, these sets yield  localization of the spectrum and appear to have interesting geometrical and topological properties.

The paper is organized as follows. After some preliminaries, Section 3 introduces the new $S-$SDD inclusion-exclusion localization sets and proves the corresponding eigenvalue inclusion theorem. In Section 4 we compare the proposed CKV inclusion-exclusion localization set with existing inclusion-exclusion localization sets, while  Section 5 contains some important remarks on plotting the CKV inclusion-exclusion set.  Section 6 deals with matrix polynomials, and the paper ends with Section 7 - Conclusions and the list of relevant references.

\section{Preliminaries}

Throughout the paper, we assume that  $A=[a_{ij}]\in\mathbb C^{n\times n}$ and use the following notation:

\begin{itemize}[itemsep=5pt]
	\item[---] $N:=\{1,2,\ldots,n\}$ for the set of indices
	\item[---] $ \emptyset\neq S\subseteq N$ for an arbitrary nonempty subset of indices
	\item[---] $\overline S:=N\setminus S$ for its complement
	\item[---] $r_i=r_i(A):=\displaystyle{\sum_{j\in N\setminus\{i\}}|a_{ij}|}$
	\item[---] $r_i^S=r_i^S(A):=\displaystyle{\sum_{j\in S\setminus\{i\}}|a_{ij}|}$
	$\quad$ and $\quad$ $ r_i^{\overline{S}}=r_i^{\overline{S}}(A):=\displaystyle{
		\sum_{j\in \overline S\setminus\{i\}}|a_{ij}|}$
\end{itemize}
For any $i\in N$, define the $i-$th Ger\v sgorin disk
\[
\Gamma_i(A):=
\left\{
z\in\mathbb C:\ |z-a_{ii}|\le r_i
\right\},
\]
and the Ger\v sgorin set
\begin{equation} \label{Ger} 
	\Gamma(A):=\bigcup_{i\in N}
	\Gamma_{i}(A).
\end{equation}
Also, for any $i,j\in N, i \neq j$,   define 
\[
{\cal D}_{ij}(A):=
\left\{
z\in\mathbb C:\ |z-a_{ii}|\Big(|z-a_{jj}|-r_j^i\Big)\le r_i|a_{ji}|
\right\},
\]
where $r_j^i=r_j-|a_{ji}|,$ and define the  Dashnic-Zusmanovich set 
\begin{equation} \label{DZ} 
	{\cal D}(A)= \bigcap_{i\in N} \bigcup_{j\in N\setminus\{i\}} {\cal D}_{ij}(A).
\end{equation}
Finally, for any \(i\in S\) and \(j\in \overline S\), define 
\begingroup
\setlength{\mathindent}{0pt}
\begin{align}
 &	\Gamma_i^S(A):=
	\left\{
	z\in\mathbb C:\ |z-a_{ii}|\le r_i^S
	\right\}, \notag \\[2mm]
	& V_{ij}^S(A):=
	\left\{
	z\in\mathbb C:
	\Big(|z-a_{ii}|-r_i^S\Big)
	\Big(|z-a_{jj}|-r_j^{\overline{S}}\Big)
	\le
	r_i^{\overline{S}}r_j^S
	\right\}, \notag \\[2mm]
	&	W_{ij}^S(A):=
		\Big\{
		z\in\mathbb C:
		\Big(\max\big\{0,|z-a_{ii}|-r_i^S\big\}\Big)
		\Big(\max\big\{0,|z-a_{jj}|-r_j^{\overline{S}}\big\}\Big)
		\le
		r_i^{\overline{S}}r_j^S
		\Big\}, \label{eq:Wij}
\end{align}
\endgroup
 and 
\begin{equation}   \label{Sset} 
	C^{S}(A) := \Bigg( \bigcup_{i\in S}
	\Gamma^{S}_{i}(A)  \Bigg) \bigcup \Bigg( \bigcup_{i\in S, j\in
		\overline{S}} V^{S}_{ij}(A)  \Bigg),
\end{equation}		
and
\begin{equation}  \label{CKV} 
	C(A):= \bigcap_{\emptyset\neq S\subseteq N} C^S(A).
\end{equation}		
It is well-known, see \cite{Ger, DZ, CKV}, that (\ref{Ger}), (\ref{DZ}) and (\ref{CKV}) are localization sets for the spectrum $\sigma(A)$ (the set of all  eigenvalues of  matrix $A$), standing in the following relation:
$$
\sigma(A)\subseteq C(A) \subseteq {\cal D}(A)\subseteq \Gamma(A).
$$

\subsection{A new characterization of the set $C^S(A)$} \label{sec:SSDDnewIncl}

Before we continue, we will give another representation of the set $C^S(A)$ defined by (\ref{Sset}). Namely, if we look carefully at the definition of $C^S(A)$, we can conclude that
for all $i\in S$ and $j\in \overline{S}$:
\[  \Gamma^{\overline{S}}_{j}(A)    \subseteq 	\Gamma^{S}_{i}(A)  \cup V^{S}_{ij}(A) .
\]
Indeed, if we take an arbitrary $z \in \Gamma^{\overline{S}}_{j}(A)$, it means that  $|z-a_{jj}|\le r_j^{\overline{S}}$. 
\begin{itemize}
	\item If $z \in \Gamma_i^S(A)$, the proof is completed.
	\item If  $|z-a_{ii}| > r_i^S$, then 
	$\Big(|z-a_{ii}|-r_i^S\Big)
	\Big(|z-a_{jj}|-r_j^{\overline{S}}\Big)$ is nonpositive, hence $z \in  V^{S}_{ij}(A) $.
\end{itemize}
Now, we immediately conclude that
\[C^S(A)= C^S_1(A):=
\Bigg( \bigcup_{i\in S} 	\Gamma^{S}_{i}(A)  \Bigg) \bigcup
\Bigg( \bigcup_{ j\in \overline{S}} 	\Gamma^{\overline{S}}_{j}(A)  \Bigg) \bigcup \Bigg( \bigcup_{i\in S, j\in\overline{S}} V^{S}_{ij}(A)  \Bigg).
\]
Precisely this form of the set $C^S(A)$ allows us to represent it in the following form:

\begin{lemma} \label{lemma:1}
	For an arbitrary matrix $A$ and an arbitrary nonempty proper subset $S$ of $N$:
	\[ C^S(A)= \bigcup_{i\in S, j\in \overline{S}} W^{S}_{ij}(A).  
	\]
\end{lemma}
\begin{proof}
	Obviously,  for all $i\in S, j\in \overline{S}$, it holds that 
	\[\Gamma^{S}_{i}(A)   \cup \Gamma^{\overline{S}}_{j}(A)  \cup V_{ij}^S(A)\subseteq W_{ij}^S(A),\] so, it remains to show that  for all $i\in S$ and $j\in \overline{S}$:
	\[
	z \in  W^{S}_{ij}(A)  
	\quad \Longrightarrow \quad
	z \in 	\Gamma^{S}_{i}(A)   \cup \Gamma^{\overline{S}}_{j}(A)  \cup V^{S}_{ij}(A).
	\]
	Let  $i\in S, j\in \overline{S}$ be such that $z \in  W^{S}_{ij}(A) $, i.e.
	\[
	\Big(\max \big\{0,|z-a_{ii}|-r_i^S \big\}\Big)
	\Big(\max\big\{0,|z-a_{jj}|-r_j^{\overline{S}}\big\}\Big)
	\le
	r_i^{\overline{S}}r_j^S .\]
	We now consider the possible cases.
	\begin{itemize}
		\item  If \; $|z-a_{ii}|-r_i^S > 0$ \, and \, $|z-a_{jj}|-r_j^{\overline{S}}>0$ $\quad \Longrightarrow \quad$ $z \in V_{ij}^S(A)$.
		\item If \; $|z-a_{ii}|-r_i^S > 0$ \, and \, $|z-a_{jj}|-r_j^{\overline{S}}\le 0$ $\quad \Longrightarrow \quad$ $z \in \Gamma^{\overline{S}}_{j}(A)$.
		\item If \; $|z-a_{ii}|-r_i^S \le 0$ \, and \, $|z-a_{jj}|-r_j^{\overline{S}}>0$ $\quad \Longrightarrow \quad$ $z \in \Gamma_{i}^S(A)$.
		\item If \; $|z-a_{ii}|-r_i^S \le 0$ \, and \, $|z-a_{jj}|-r_j^{\overline{S}}\le 0$ $\quad \Longrightarrow \quad$ $z \in \Gamma_i^S(A) \bigcap \Gamma^{\overline{S}}_{j}(A)$.
	\end{itemize}
	The proof is completed.
\end{proof}
However, it is interesting to note that the  set $C^S(A)$, for an arbitrary nonempty proper subset $S$ of $N$, can be represented as an intersection-based construction rather than as a union of ``small'' Ger\v sgorin disks:
\[ C^S(A)=C^S_2(A):=\Bigg( \bigcap_{i\in S} 	\Gamma^{S}_{i}(A)  \Bigg) \bigcap
\Bigg( \bigcap_{ j\in \overline{S}} 	\Gamma^{\overline{S}}_{j}(A)  \Bigg) \bigcup \Bigg( \bigcup_{i\in S, j\in\overline{S}} V^{S}_{ij}(A)  \Bigg).
\]
In order to prove that, it is sufficient to show that 
\[
\bigcup_{i\in S, j\in\overline{S}} 	\Big(\Gamma^{S}_{i}(A)   \bigcup
\Gamma^{\overline{S}}_{j}(A)\Big)    \subseteq
C^S_2(A).
\]
Take $z \in 	\bigcup_{i\in S, j\in\overline{S}} 	\Big(\Gamma^{S}_{i}(A)   \bigcup
\Gamma^{\overline{S}}_{j}(A)\Big)  $. This means that there exist 	
$i\in S$, $j\in \overline{S}$, such that
\[
|z-a_{ii}|\le r_i^S \; \vee \; |z-a_{jj}|\le r_j^{\overline{S}}  .
\]
Then, there are three possibilities:
\begin{itemize}
	\item $|z-a_{ii}|-r_i^S \le 0$ \; and \; $|z-a_{jj}|-r_j^{\overline{S}}>0$ $\quad \Longrightarrow \quad$ $z \in V_{ij}^S(A)$,
	\item $|z-a_{ii}|-r_i^S > 0$ \; and \; $|z-a_{jj}|-r_j^{\overline{S}}\le 0$ $\quad \Longrightarrow \quad$ $z \in V^{S}_{ij}(A)$,
	\item $|z-a_{ii}|-r_i^S \le 0$ \; and \; $|z-a_{jj}|-r_j^{\overline{S}}\le 0$ $\quad \Longrightarrow \quad$ $z \in \Gamma_i^S(A) \bigcap \Gamma^{\overline{S}}_{j}(A)$. 
		
	\noindent If there exists $p\in S$ such that
	$z\notin\Gamma_p^S(A)$, then, since $z\in\Gamma_j^{\overline S}(A)$, we have $z\in V_{pj}^S(A)$. If there exists $q\in\overline S$ such that
	$z\notin\Gamma_q^{\overline S}(A)$,  then, since $z\in\Gamma_i^S(A)$, we have $z\in V_{iq}^S(A)$. Finally, if there are no such indices $p$ and $q$, i.e. if for all $p\in S$ and $q\in\overline S$: $z\in\Gamma_p^S(A)$ and $z\in\Gamma_q^{\overline S}(A)$, then $z\in \left(\bigcap_{p\in S}\Gamma_p^S(A)\right)\cap \left(\bigcap_{q\in\overline S}\Gamma_q^{\overline S}(A)\right)$.
\end{itemize}
Therefore, in all cases,
$z\in C_2^S(A)$.
If we denote 
\[ \Gamma^S_*(A) :=\bigcap_{i\in S, j\in \overline{S}} \Big(	\Gamma^{S}_{i}(A)\bigcap \Gamma^{\overline{S}}_{j}(A) \Big)\]
the following form of the  set $C^S(A)$:
\begin{equation} \label{ekv2}C^S(A)=C^S_2(A):=\Gamma^S_*(A) \bigcup\Bigg( \bigcup_{i\in S, j\in\overline{S}}   V^{S}_{ij}(A)  \Bigg).
\end{equation}
shows that it is sufficient to plot only the sets $V_{ij}^S(A) $ if there is no intersection of all ``small'' Ger\v sgorin disks.

\subsection{Ger\v sgorin and Dashnic-Zusmanovich inclusion-exclusion set}

For \(i,k\in N,\ i\neq k\), define
\[
\Delta_{ik}(A)
=
\left\{z\in\mathbb C:\ |z-a_{kk}|<2|a_{ki}|-r_k\right\},
\quad \text{and} \quad
\Delta_i(A)=\bigcup_{k\in N\setminus\{i\}}\Delta_{ik}(A).
\]
Then the Ger\v sgorin inclusion-exclusion localization set is given by
\[
\Omega(A)=\bigcup_{i\in N}\Big(\Gamma_i(A)\setminus\Delta_i(A)\Big).
\]
According to Theorem 2 in \cite{Mel},
$
\sigma(A)\subseteq\Omega(A)\subseteq\Gamma(A).
$

In \cite{ZS}, 
the corresponding Dashnic-Zusmanovich exclusion set is defined by
\[
{\cal L}_{ij}(A)=\Big\{z\in\mathbb C:|z-a_{ii}|\Big(|z-a_{jj}|+r_j^i\Big)
<\Big(2|a_{ij}|-r_i\Big)|a_{ji}|\Big\},
\]
and the Dashnic-Zusmanovich (DZ) inclusion-exclusion localization set is given by
\[
\Theta(A) =
\bigcap_{i\in N}
\bigcup_{j\in N\setminus\{i\}}
\Big( {\cal D}_{ij}(A)\setminus {\cal L}_{ij}(A) \Big).
\]
According to Theorems 5 and 6 in \cite{ZS},
$
\sigma(A)\subseteq\Theta(A)\subseteq{\cal D}(A)\subseteq \Gamma(A).
$

\section{$S-$SDD inclusion-exclusion sets} \label{sec:SSDDinclexcl}

In this section, we will derive new inclusion-exclusion localization sets obtained by excluding certain subsets from the corresponding CKV localization sets. 

For any $i\in S$ and $j\in\overline S$, define
\begin{equation}\label{eq:psi}
	\begin{aligned}
		\Psi_{ij}^S(A):=
		\Bigg\{
		z\in\mathbb C&:		
		\Big(|z-a_{ii}|+r_i^S\Big)
		\Big(|z-a_{jj}|+r_j^{\overline{S}}\Big)<\\
		&<\Big(\max\big\{0,2|a_{ij}|-r_i^{\overline{S}}\big\}\Big)
		\Big(\max\big\{0,2|a_{ji}|-r_j^S\big\}\Big)
		\Bigg\}.
	\end{aligned}
\end{equation}
\begin{theorem}\label{glavna}
	Let \(A=[a_{ij}]\in\mathbb C^{n\times n}\), and let \(S\) be any nonempty proper subset of
	$N$. Then,
	\[
	\sigma(A)
	\subseteq   \widehat{\Xi}^{S}(A) :=  \bigcup_{i\in S, j\in 
		\overline{S}}  \Big(W_{ij}^S(A) \setminus \Psi_{ij}^S(A)\Big) .
	\]
\end{theorem}
\begin{proof}
	Choose an arbitrary eigenvalue \(\lambda\) of the matrix \(A\).
	Then there exists a nonzero vector \(x\) such that $Ax=\lambda x.$ 
	
	Let $k$ be an index for which
	$|x_k|=\|x\|_\infty >0 .$ Suppose that $k \in S$ and denote by $\ell$ an index from $\overline{S}$, for which
	$ |x_\ell |=\max_{i \in \overline{S}}|x_i| .$		Since for all \(i\in N\),
	\[
	(\lambda-a_{ii})x_i=
	\sum_{j\in S\setminus\{i\}}a_{ij}x_j
	+
	\sum_{j\in\overline S\setminus\{i\}}a_{ij}x_j,
	\]
	and consequently
	\[
	|\lambda-a_{ii}||x_i| 	\le 	r_i^S|x_k| + 	r_i^{\overline{S}}|x_\ell|,
	\] 
	we have that
	\[
	\Big(|\lambda-a_{kk}|-r_k^S\Big)|x_k|
	\le 	r_k^{\overline{S}}|x_\ell|
	\quad \mbox{and} \quad 
	\Big(|\lambda-a_{\ell\ell}|-r_\ell^{\overline{S}}\Big)|x_\ell|
	\le 	r_\ell^S|x_k|.
	\]
	Therefore,
	\begin{equation}
		\max\big\{0,|\lambda-a_{kk}|-r_k^S\big\}|x_k|
		\le 	r_k^{\overline{S}}|x_\ell|,
		\label{eq:firstineq}
	\end{equation}	and
	\begin{equation}
		\max\big\{0,|\lambda-a_{\ell\ell}|-r_\ell^{\overline{S}}\big\}|x_\ell|
		\le
		r_\ell^S|x_k|.
		\label{eq:secondineq}
	\end{equation}
	
	\begin{itemize}
		\item \textbf{Case 1.} 
		Suppose that \( |x_\ell|>0 \).
		
		Multiplying inequalities \eqref{eq:firstineq} and \eqref{eq:secondineq}
		and dividing by \(|x_k||x_\ell|\), we obtain
		\begin{equation*}
			\max\big\{0,|\lambda-a_{kk}|-r_k^S\big\}
			\max\big\{0,|\lambda-a_{\ell\ell}|-r_\ell^{\overline{S}}\big\}
			\le
			r_k^{\overline{S}}r_\ell^S,
		\end{equation*}
		\begin{equation}
			\mbox{i.e.} \;\;
			\lambda\in W_{k\ell}^S(A).
			\label{V}
		\end{equation}
		On the other hand, for all \(i\in N\) and all
		\(t\in N\setminus\{i\}\),
		\[ a_{it}x_t= (\lambda -a_{ii})x_i -\sum_{j\in \overline{S}\setminus\{i,t\}}a_{ij}x_j-  \sum_{j\in S\setminus\{i,t\}}a_{ij}x_j .\]
		Hence,
		\begin{equation*}
			\begin{aligned}
				|a_{\ell k}||x_k| &\le |\lambda -a_{\ell \ell}||x_\ell| +\sum_{j\in \overline{S}\setminus\{\ell\}}|a_{\ell j}||x_\ell|+  \sum_{j\in S\setminus\{k\}}|a_{\ell j}||x_k|\\
				&= |\lambda -a_{\ell \ell}||x_\ell| +r_\ell^{\overline{S}}|x_\ell|+\Big(r_\ell^S-|a_{\ell k}|\Big)|x_k|,\\[0.2cm]
				|a_{k\ell}||x_\ell| &\le |\lambda -a_{kk}||x_k| +\sum_{j\in \overline{S}\setminus\{\ell\}}|a_{k j}||x_\ell|+  \sum_{j\in S\setminus\{k\}}|a_{k j}||x_k|\\
				&=   |\lambda -a_{kk}||x_k| +\Big(r_k^{\overline{S}}-|a_{k\ell}|\Big)|x_\ell|+r_k^S|x_k|,
			\end{aligned}
		\end{equation*}
		i.e.
		\[ \Big(2|a_{\ell k}|-r_\ell^S\Big)|x_k| \le  \Big(|\lambda -a_{\ell \ell}| +r_\ell^{\overline{S}}\Big)|x_\ell| \]
		and \[\Big(2|a_{k\ell}|-r_k^{\overline{S}}\Big)|x_\ell| \le   \Big(|\lambda -a_{kk}|+r_k^S\Big)|x_k|.\]
		Therefore,
		\begin{equation}
			\max\Big\{0,2|a_{\ell k}|-r_\ell^S\Big\}|x_k| \le  \Big(|\lambda -a_{\ell \ell}| +r_\ell^{\overline{S}}\Big)|x_\ell|,
			\label{eq1:max_other hand}
		\end{equation}and
		\begin{equation}
			\max\Big\{0,2|a_{k\ell}|-r_k^{\overline{S}}\Big\}|x_\ell| \le   \Big(|\lambda -a_{kk}|+r_k^S\Big)|x_k|.
			\label{eq2:max_other hand}
		\end{equation}
		Multiplying inequalities (\ref{eq1:max_other hand}) and (\ref{eq2:max_other hand}) and dividing by
		\(|x_k||x_\ell|\), we obtain
		\[	\max\Big\{0,2|a_{\ell k}|-r_\ell^S\Big\}  \max\Big\{0,2|a_{k\ell}|-r_k^{\overline{S}}\Big\} 
				\le \Big(|\lambda -a_{\ell \ell}| +r_\ell^{\overline{S}}\Big) \Big(|\lambda -a_{kk}|+r_k^S\Big), \]
\begin{equation}
			\mbox{i.e.} \;\;
			\lambda \notin \Psi_{k\ell}^S(A).
			\label{W}
		\end{equation}	
		Finally, (\ref{V}) and (\ref{W}) mean that $	\lambda\in
		\Big(W_{k\ell}^S(A)\setminus \Psi_{k\ell}^S(A)\Big).$
		
		\item \textbf{Case 2.} Suppose that  \( |x_\ell|=0 \). This means that  $x_j =0$ for all $j \in \overline{S}$, and for all $ i \in N $
		$$ (\lambda -a_{ii})x_i= \sum_{j\in S\setminus\{i\}}a_{ij}x_j.$$
		Hence,
		\begin{equation*} |\lambda -a_{kk}||x_k| \le r_k^S |x_k|, \;\;
			\mbox{i.e.} \;\; 
			\lambda \in \Gamma_k^S(A).
			\label{g} 
		\end{equation*} 
		On the other hand,  we have:
		\[ a_{\ell k}x_k= (\lambda -a_{\ell \ell})x_\ell -  \sum_{j\in S\setminus\{k\}}a_{\ell j}x_j = -  \sum_{j\in S\setminus\{k\}}a_{\ell j}x_j .\]
		Consequently,
		\[ |a_{\ell k}||x_k|\le   \Big( r_\ell^S-|a_{\ell k}|\Big) |x_k| 
		\quad \mbox{and} \quad 2|a_{\ell k}|\le   r_\ell^S .\]
		Therefore, $\max\big\{0, 2|a_{\ell k}|- r_\ell^S\big\}=0$, hence 
		$\lambda \notin  \Psi_{k\ell}^S(A)$. 
			Since $\Gamma_{k}^S(A) \subseteq W_{k\ell}^S(A)$, we conclude that $\lambda \in \Big(W_{k\ell}^S(A) \setminus \Psi_{k\ell}^S(A)\Big)$.
	\end{itemize}
	This  completes the proof if $k \in S$. However, the case \(k\in\overline S\) does not need to be considered separately, since for any $\emptyset \ne S \subsetneq   N$,
	due to symmetry, we immediately have
	\[
	W_{ij}^S(A)\setminus \Psi_{ij}^S(A)
	=
	W_{ji}^{\overline{S}}(A)\setminus \Psi_{ji}^{\overline{S}}(A),
	\]
	for all $i \in S, j \in {\overline{S}}$, and consequently
	$\widehat{\Xi}^{S}(A)=\widehat{\Xi}^{\overline{S}}(A) .$
	The proof is completed.
\end{proof}

\begin{corollary}\label{posl}
	Let \(A=[a_{ij}]\in\mathbb C^{n\times n}\), and let \(S\) be any  nonempty subset of $N$. Then
	\[
	\sigma(A)
	\subseteq   \Xi^{S}(A)=
	\left\{ 
	\begin{array}{lll}
		\widehat{\Xi}^{S}(A) & \mbox{if} & \emptyset \neq S \subsetneq N\\
		\Omega(A)& \mbox{if} &  S = N
	\end{array}. \right.
	\]
\end{corollary}

\begin{proof}
	The case $S=N$ has been proved in  \cite{Mel}, while the case $\emptyset \neq S \subsetneq N$ is covered by Theorem \ref{glavna}.
\end{proof}

We end this section with an observation that the best possible choice of inclusion-exclusion localization can be obtained by intersection over all possible choices of $S$. More precisely, using Corollary \ref{posl}, we can define 
the corresponding CKV inclusion-exclusion localization set by
\[
\Xi(A):=
\bigcap_{\emptyset\neq S\subseteq N}
\Xi^S(A).
\]

\section{Comparisons with existing inclusion-exclusion localization sets}

In this section, we compare the proposed CKV inclusion-exclusion localization set with existing inclusion-exclusion localization sets. 

According to Corollary  \ref{posl}, the Ger\v sgorin inclusion-exclusion localization set is a special case
of our CKV inclusion-exclusion localization set  for $S=N$.

Now, we will show that the Dashnic-Zusmanovich inclusion-exclusion localization set is also a special case of our CKV inclusion-exclusion localization set. In fact, the DZ inclusion-exclusion localization set arises as a special case of our construction when intersections are taken over all singleton subsets \(S=\{i\}\), \(i\in N\). 

Indeed, for \(S=\{i\}\), we have
\[
r_i^S=0,
\quad
r_j^S=|a_{ji}|,
\quad
r_i^{\overline{S}}=r_i,
\quad
r_j^{\overline{S}}=r_j-|a_{ji}|, \quad
\mbox{for all} \;\;  j\in N\setminus\{i\},
\]
 \begin{equation*}
		\begin{aligned}
			W_{ij}^S(A)
			&=
			\Big\{
			z\in\mathbb C:
			|z-a_{ii}|
			\max\big\{0,|z-a_{jj}|-r_j+|a_{ji}|\big\}
			\le
			r_i|a_{ji}|
			\Big\}\\
			&=\Big\{
			z\in\mathbb C:
			|z-a_{ii}|
			\Big(|z-a_{jj}|-r_j+|a_{ji}|\Big)
			\le
			r_i|a_{ji}|
			\Big\}=
			{\cal D}_{ij}(A),\\[0.2cm]
			\Psi_{ij}^S(A)
			& =
			\Big\{
			z\in\mathbb C:
			|z-a_{ii}|
			\Big(
			|z-a_{jj}|+r_j-|a_{ji}|
			\Big) <\Big(\max\big\{0,2|a_{ij}|-r_i\big\}\Big)|a_{ji}|
			\Big\}\\
			& =
			\Big\{
			z\in\mathbb C:
			|z-a_{ii}|
			\Big(
			|z-a_{jj}|+r_j-|a_{ji}|
			\Big)<
			\Big(2|a_{ij}|-r_i\Big)|a_{ji}|
			\Big\}=
			{\cal L}_{ij}(A),
		\end{aligned}
\end{equation*}
and 
\[W_{ij}^S(A)\setminus \Psi_{ij}^S(A) =
{\cal D}_{ij}(A) \setminus {\cal L}_{ij}(A).
\]
Hence,
\[
\Xi^{\{i\}}(A)
=
\bigcup_{j\in N\setminus\{i\}}
\Big(
{\cal D}_{ij}(A)\setminus {\cal L}_{ij}(A)
\Big).
\]
Taking the intersection over all singleton subsets \(S=\{i\}\), \(i\in N\), we obtain
\[
\bigcap_{i\in N}\Xi^{\{i\}}(A)
=
\Theta(A).
\]
Obviously, 
\[
\Xi(A)=
\bigcap_{\emptyset\neq S\subseteq N}
\Xi^S(A)
\subseteq
\bigcap_{i\in N}\Xi^{\{i\}}(A)
=
\Theta(A).
\]

These relations show that the proposed inclusion-exclusion localization set can be viewed as a natural generalization of the Ger\v sgorin and Dashnic-Zusmanovich inclusion-exclusion localization sets. In fact, for an arbitrary matrix $A$, its spectrum is contained in $\Xi(A)$, for which 
$$\begin{array}{ccccc}
	\Xi(A) &\subseteq & \Theta(A)& & \Omega(A)\\
	\mbox{\rotatebox[]{-90}{$\subseteq$}} & 
	&\mbox{\rotatebox[]{-90}{$\subseteq$}} & &\mbox{\rotatebox[]{-90}{$\subseteq$}} \\
	C(A) &\subseteq & {\cal D}(A)&\subseteq &\Gamma(A) \\
\end{array}.
$$
It is interesting to note that $\Theta(A) \nsubseteq \Omega(A)$. To illustrate this, we consider the matrix from \cite{MPP}.
\begin{example}
	For the $4 \times 4$ matrix
	$$ 
	A = 
	\left[\begin{array}{rrrr}
		5 & 4 & -1 & 0  \\
		5 & 2 & 0 & 1  \\
		1& -1 & -4 & 1  \\
		1 & 1 & 0 & 0.2  \\
	\end{array}\right],
	$$
	the sets $\Omega(A)$ and $\Theta(A)$ are incomparable by inclusion, see Figure \ref{dzg}.
	\begin{figure}[ht]
		\centering
		\begin{minipage}{0.55\textwidth}
			\centering
			\includegraphics[width=\linewidth]{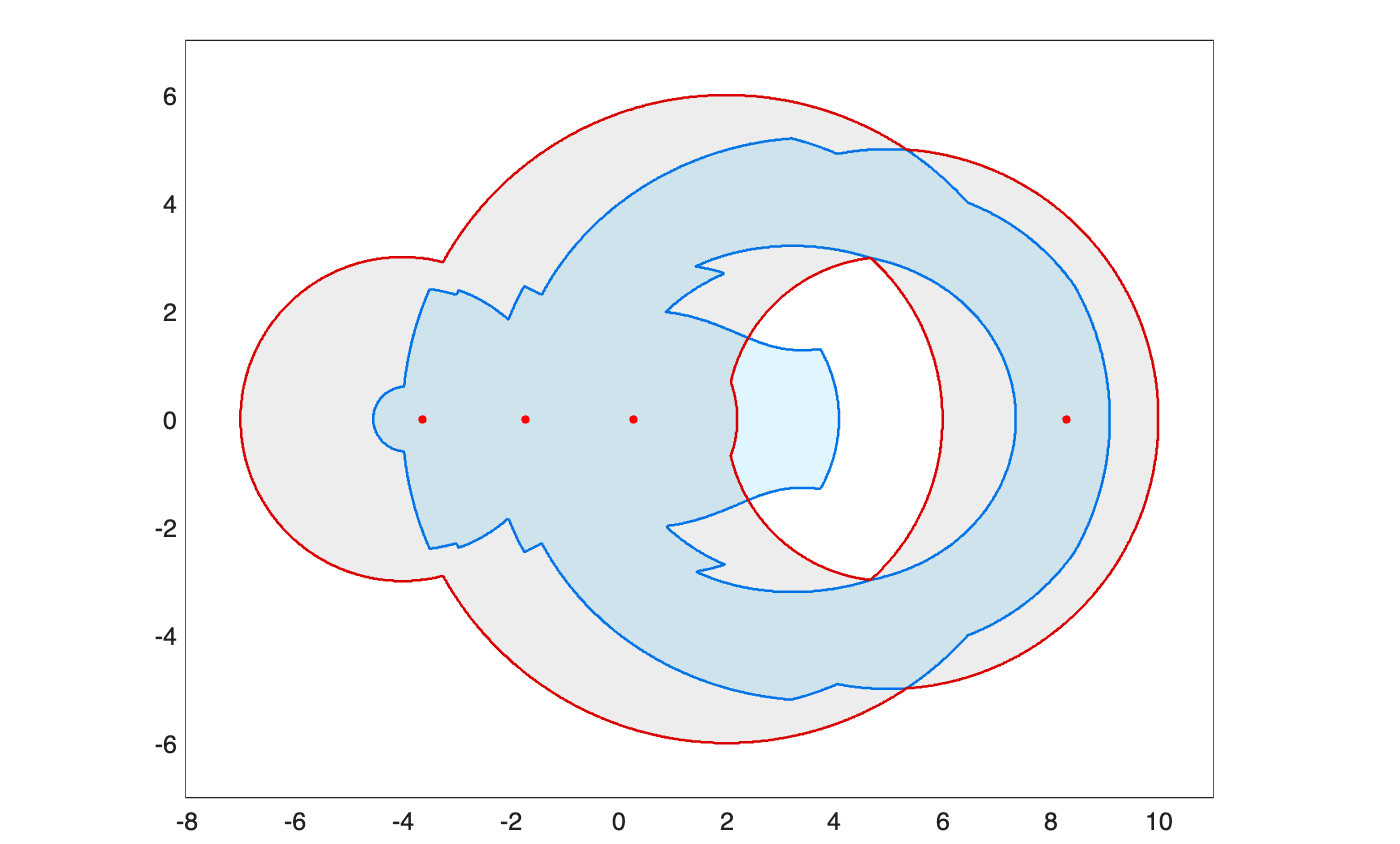}
			\caption{Sets $\Omega(A)$ (gray, red boundary) and
				$\Theta(A)$ (light blue, blue boundary).}
			\label{dzg}
		\end{minipage}
	\end{figure}

\end{example}


\section{Remarks on plotting $\Xi^S(A)$ sets}

The definition of the set $\Xi^S(A)$ given in Corollary \ref{posl} was suitable from a theoretical point of view,  and now we will derive another equivalent definition,  which follows from the equivalent form $C^S_2(A)$ of the $C^S(A)$ set given in (\ref{ekv2}):
\[
C^S(A)=C^S_2(A):=\Gamma^S_*(A) \bigcup\Bigg( \bigcup_{i\in S, j\in\overline{S}}   V^{S}_{ij}(A)  \Bigg),
\]
from one side, and a closer look at the $\Psi_{ij}^S(A)$ sets, on the other side.

Since the case $S=N$ reduces to Ger\v sgorin, and $S$ being a singleton reduces to Dashnic-Zusmanovich, let us suppose that both $S$ and $\overline{S}$ contain at least two indices. 
If we define
\[
\Upsilon_{ij}^S(A):=
\Bigg\{
z\in\mathbb C:		
\Big(|z-a_{ii}|+r_i^S\Big)
\Big(|z-a_{jj}|+r_j^{\overline{S}}\Big)
<\Big(2|a_{ij}|-r_i^{\overline{S}}\Big)
\Big(2|a_{ji}|-r_j^S\Big)
\Bigg\},
\]
then, for any pair $i \in S$, $j \in \overline{S}$, it holds that
\[\Psi_{ij}^S(A)=\left\{
\begin{array}{lll}
	\emptyset  & \mbox{if} & 2|a_{ij}|\le r_i^{\overline{S}}\;\; \mbox{and} \;\; 2|a_{ji}|> r_j^{{S}}\\
	\emptyset & \mbox{if} & 2|a_{ij}| > r_i^{\overline{S}}\;\; \mbox{and} \;\; 2|a_{ji}|\le r_j^{{S}}    \\
	\Upsilon_{ij}^S(A)  & \mbox{if} & 2|a_{ij}| > r_i^{\overline{S}} \;\; \mbox{and} \;\; 2|a_{ji}|> r_j^{{S}}\\
	\emptyset  & \mbox{if} & 2|a_{ij}| \le r_i^{\overline{S}} \;\; \mbox{and} \;\; 2|a_{ji}|\le r_j^{{S}}
\end{array}.
\right.
\]
It is easy to see that for every $i \in S$ there exists at most one index $j \in \overline{S}$ for which  $2|a_{ij}| > r_i^{\overline{S}}$. Indeed, if we suppose that there exist two distinct indices, $j \in \overline{S}$ and $k \in \overline{S}$, such that  
\[2|a_{ij}| > r_i^{\overline{S}} \; \wedge \;  2|a_{ik}| > r_i^{\overline{S}} ,\]
adding these inequalities gives us
\[2\big(|a_{ij}| +|a_{ik}| \big)>2r_i^{\overline{S}}  ,\]
which is an obvious contradiction.

Also, for every  $j \in \overline{S}$ there exists at most one index $i \in S$ for which  $2|a_{ji}| > r_j^S$. Consequently, there exists at most one pair
$(k,\ell)\in S\times \overline{S}$
such that
$2|a_{k\ell}|>r_k^{\overline{S}}$
and
$2|a_{\ell k}|>r_\ell^S.$ Hence, 
\[\bigcup_{ i\in S, j\in \overline{S}}   \Psi_{ij}^S(A)=
\left\{
\begin{array}{ll}
	\Psi_{k\ell}^S(A)= \Upsilon_{k\ell}^S(A) & \mbox{if there exist} \; k \in S ,\ell   \in \overline{S}: 
	2|a_{k\ell}| > r_k^{\overline{S}}\\&  \mbox{and} \;\; 2|a_{\ell k}|> r_\ell^{{S}}\\[0.2cm]
	\emptyset & \mbox{otherwise}
\end{array}.
\right.
\]
Having said all of this, the equivalent form of the $S-$SDD inclusion-exclusion set is
\[
\widehat{\Xi}^{S}(A)=\Bigg(\Gamma^S_*(A) \bigcup\Big( \bigcup_{i\in S, j\in\overline{S}}   V^{S}_{ij}(A)\Big)  \Bigg) \setminus \Psi_{k\ell}^S(A),
\]
where $k \in S ,\ell   \in \overline{S}$ denote the unique pair of indices (if such a pair exists) satisfying
\[2|a_{k\ell}| > r_k^{\overline{S}} \;\; \mbox{and} \;\; 2|a_{\ell k}|> r_\ell^{{S}} .\]

To conclude, we present the sequence of steps in drawing  the $S-$SDD inclusion-exclusion region:
\begin{enumerate}[label=\arabic*.]
	\item 
	Choose $S \subset N: 2\le |S| \le n-2$.
	\item Check if 	$\Gamma_*^S(A)$ is nonempty. If yes, set
	$\widehat{\Xi}^S(A)=\Gamma_*^S(A)$, otherwise set $\widehat{\Xi}^S(A)=\emptyset$.
	\item
	$\widehat{\Xi}^S(A)=\widehat{\Xi}^S(A)\cup
	\left(\bigcup_{i\in S,j\in\overline S}V_{ij}^S(A)\right)$
	\item
	Check  if there exist indices $k \in S ,\ell   \in \overline{S}$ such that $2|a_{k\ell}| > r_k^{\overline{S}} $ and  $2|a_{\ell k}|> r_\ell^{{S}} $ . 
	\item 
	If such indices exist, set	 $\widehat{\Xi}^S(A)=\widehat{\Xi}^S(A)\setminus \Upsilon_{k\ell}^S(A)$.
	\vspace*{0.2cm}	
	\item
	Plot $\widehat{\Xi}^S(A)$.
\end{enumerate}

\begin{example} \em
	Consider the $5 \times 5$ matrix
	\[
	A = 
	\left[\begin{array}{rrrrr}
		-11 & -3 & 1 & -2 & -1 \\
		-1 & -10 & 0 & -1 & 3 \\
		0 & 0 & 1 & 7 & -1 \\
		1 & 0 & -7 & 7 & 0 \\
		-1 & -1 & 0 & 3 & 12
	\end{array}\right]
	.\]
	\begin{figure}[htbp]
		\centering
		
		\begin{minipage}[t]{0.48\textwidth}
			\centering
			\includegraphics[width=\linewidth]{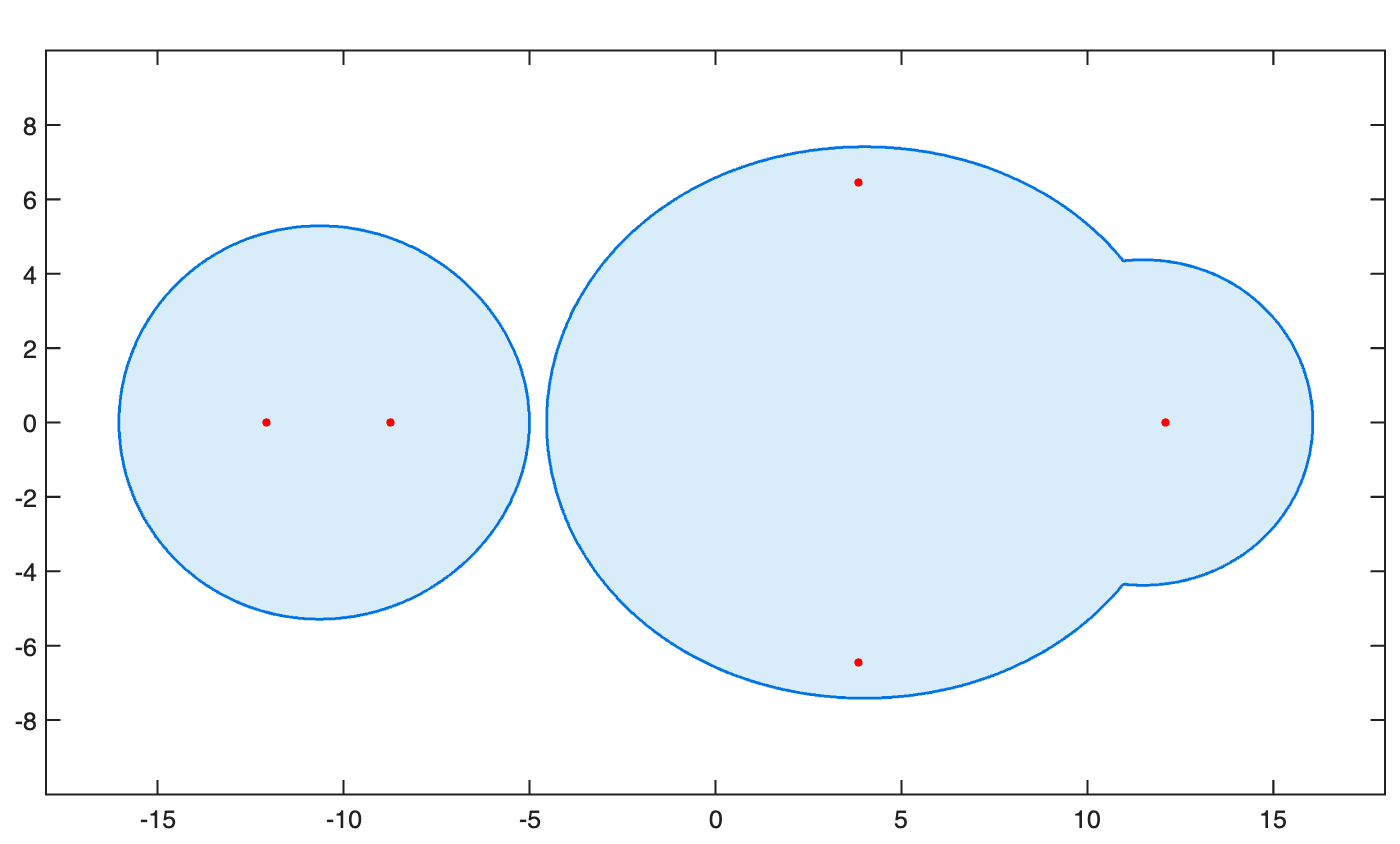}
			
			\captionof{figure}{CKV set $C^S(A)$, $S=\{1,2,3\}$.}
			\label{fig:ckvS}
		\end{minipage}
		\hfill
		\begin{minipage}[t]{0.48\textwidth}
			\centering
			\includegraphics[width=\linewidth]{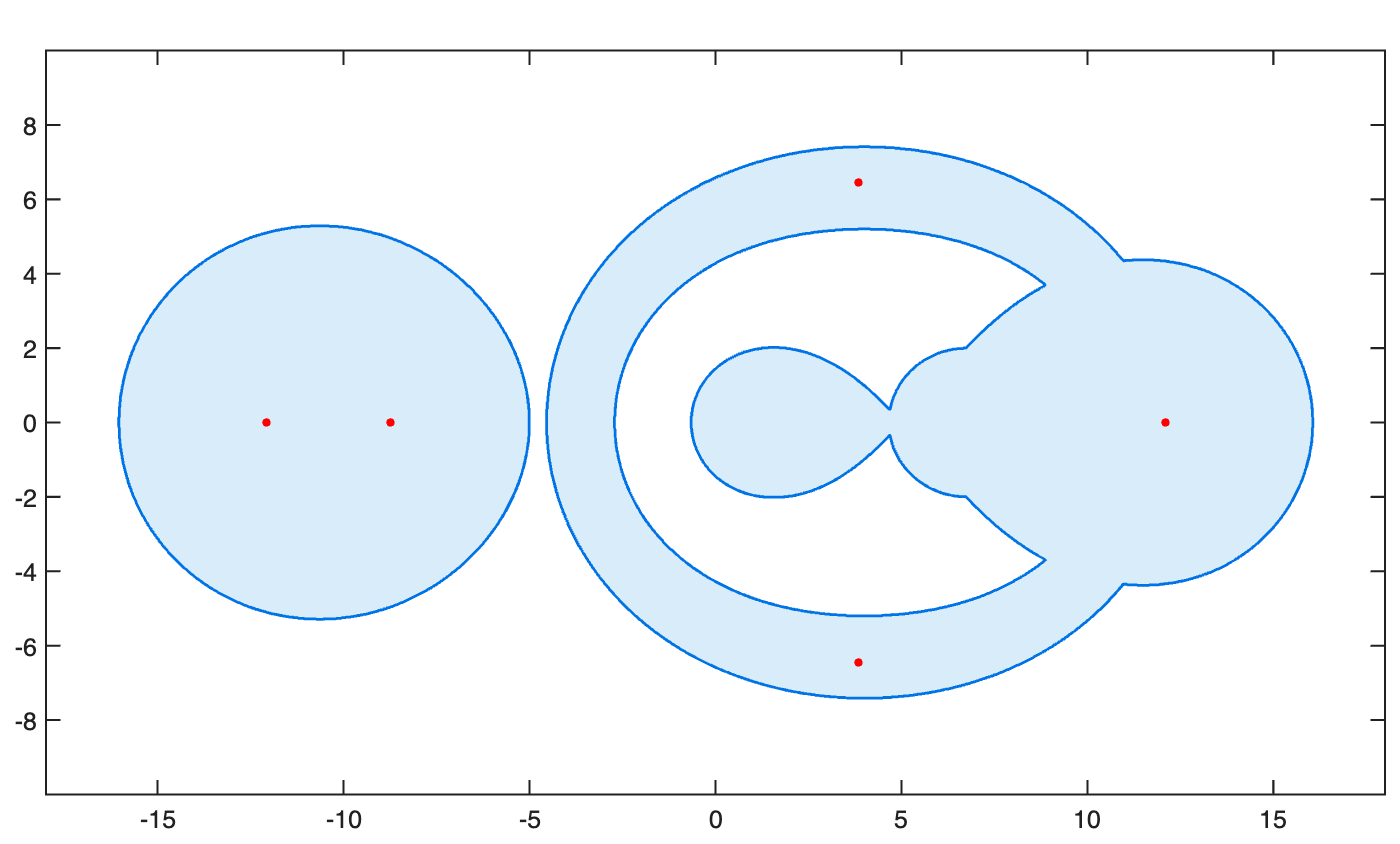}
			
			\captionof{figure}{CKV incl-excl set $\widehat{\Xi}^{S}(A)$, $S=\{1,2,3\}$.}
			\label{fig:EhatS}
		\end{minipage}
		
	\end{figure}

\begin{figure}[htbp]
	\centering
	
	\begin{minipage}[t]{0.48\textwidth}
		\centering
		\includegraphics[width=\linewidth]{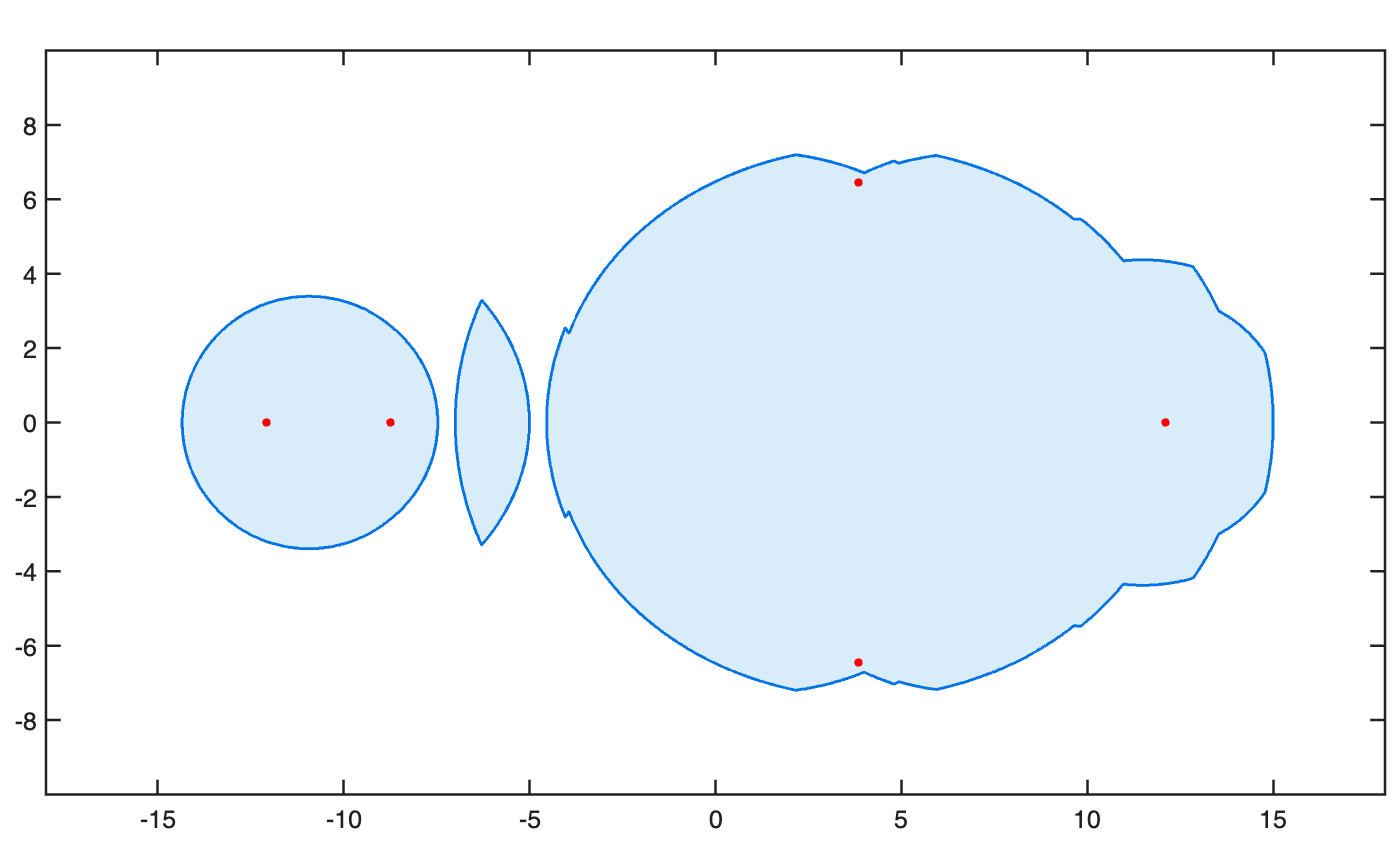}
		\captionof{figure}{CKV set $C(A)$.}
		\label{fig:ckv}
	\end{minipage}
	\hfill
	\begin{minipage}[t]{0.48\textwidth}
		\centering
		\includegraphics[width=\linewidth]{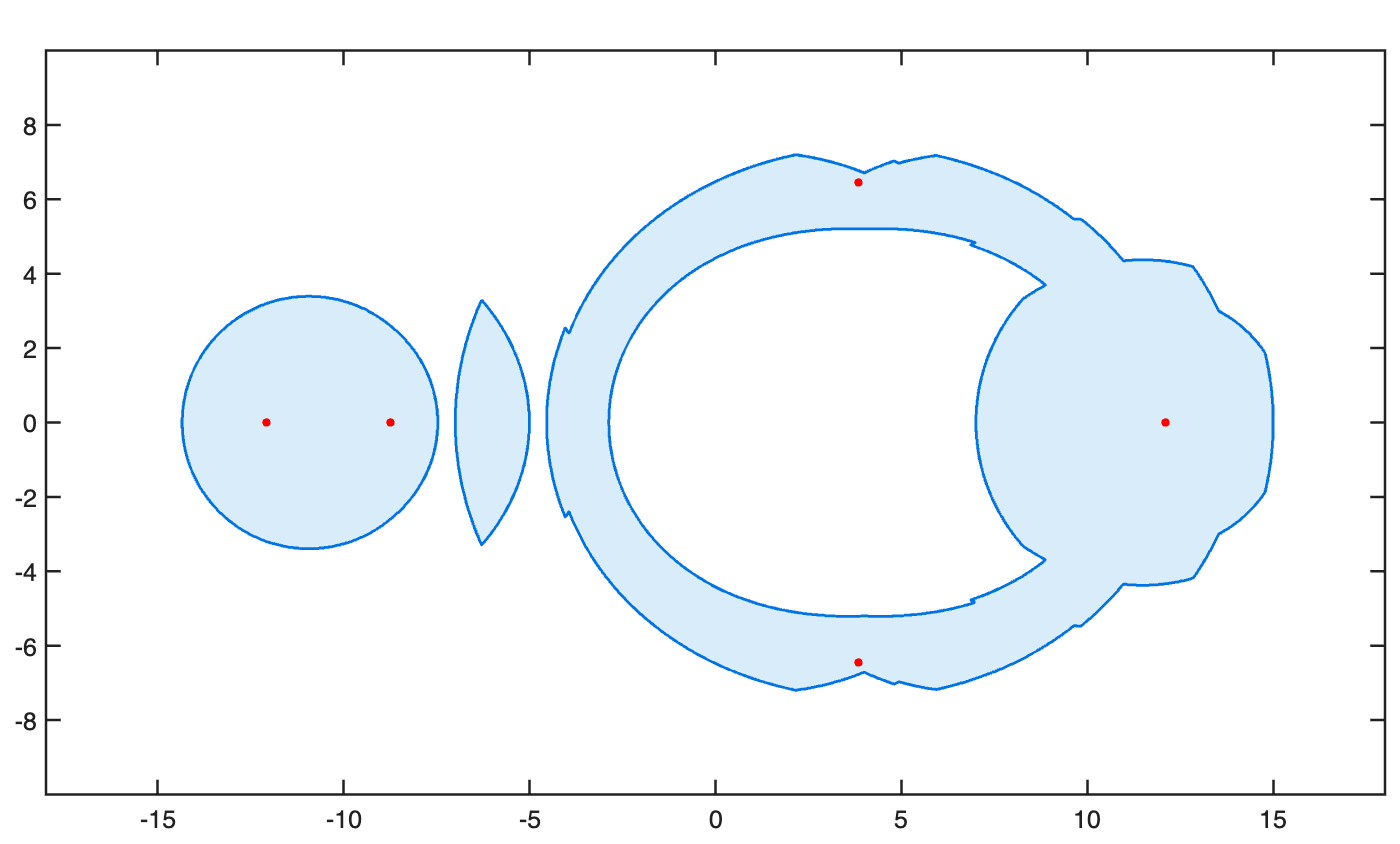}
		\captionof{figure}{CKV incl-excl set $\Xi(A)$.}
		\label{fig:E}
	\end{minipage}
	
\end{figure}

\end{example}


\section{Polynomial case}

We begin this section by providing some basic definitions and preliminary results for matrix polynomials. Subsequently, we generalize the $S-$SDD inclusion sets (defined for constant matrices in the Preliminaries) and the $S-$SDD inclusion-exclusion sets (defined in Section~\ref{sec:SSDDinclexcl}) to the case of matrix polynomials.

Consider an $n \times n$ matrix polynomial of the form
\begin{equation} \label{poly}
	P(\la) = A_m \la^m + A_{m-1} \la^{m-1} + \cdots + A_1 \la + A_0 ,
\end{equation}
where $\la$ is a complex variable, $A_0, A_1, \dots, A_m \in \mathbb{C}^{n \times n}$ with $A_m \neq 0$, and the determinant $\det P(\la)$ is not identically zero. Then, the set of all finite eigenvalues of $P(\la)$, defined by
\[
\sigma (P(\la))
= \left \{ \mu\in\Complex :\, \det P(\mu) = 0 \right \}
= \left \{ \mu\in\Complex :\, 0 \in \sigma (P(\mu)) \right \},
\]
where $\sigma (P(\mu))$ denotes the standard spectrum of the constant matrix $P(\mu)$, is the finite spectrum of $P(\la)$.
The algebraic multiplicity of an eigenvalue $\mu\in\sigma(P(\la))$ is defined as the multiplicity of $\mu$ as a root of the polynomial $\det P(\la)$, and is always at least the geometric multiplicity of $\mu$, which is given by the dimension of the null space of the matrix $P(\mu)$.
Recall that $\mu = \infty$ is an eigenvalue of $P(\la)$ if and only if $0$ is an eigenvalue of the reverse matrix polynomial $\widehat{P}(\la) = \la^m P(1/\la) = A_0 \la^m + A_1 \la^{m-1} + \cdots + A_{m-1} \la + A_m$, which is equivalent to the leading coefficient matrix $A_m$ being singular. In this case, the algebraic and geometric multiplicities of $\mu=\infty$ for $P(\la)$ are defined respectively as the algebraic and geometric multiplicities of $0$ for $\widehat{P}(\la)$.

Since the Ger\v sgorin, Brauer, and Dashnic-Zusmanovich sets for a matrix polynomial were introduced in \cite{MPP}, we only restate them briefly below before proceeding to the definition of the $S-$SDD sets.

Consider the matrix polynomial $P(\la)$ defined in (\ref{poly}) and the nonnegative functions
\begin{gather*}
	r_i(P(\lambda)) = \sum_{\substack{ j \in N \setminus \{i\} }} |(P(\lambda))_{ij}|, \quad i \in N, \\
	r_{t}^k (P(\lambda)) = r_{t} (P(\lambda)) - |(P(\lambda))_{tk}|, \quad k, t \in N, k \neq t .
\end{gather*}

\noindent The Ger\v sgorin inclusion-exclusion localization set is defined as
\begin{align*}
	& \Omega(P(\lambda)) = \bigcup_{i \in {N}} \Omega_i(P(\lambda)), \quad \text{where} \\[3pt]
	& \Omega_i(P(\lambda)) = \Gamma_i(P(\lambda)) \setminus  \Delta_i(P(\lambda)), \\[3pt]
	& \Gamma_i(P(\lambda)) = \big\{\mu \in \mathbb{C} : |(P(\mu))_{ii}| \le r_{i}(P(\mu))\big\}, \\[3pt]
	& \Delta_i(P(\lambda)) = \bigcup\limits_{ j \in N \setminus \{i\} } \Delta_{ij}(P(\la)), \\[3pt]
	& \Delta_{ij}(P(\la)) = \left \{ \mu  \in \mathbb{C} :\, |(P(\mu))_{jj}| < 2 |(P(\mu))_{ji}|- r_{j} (P(\mu)) \right \}.
\end{align*}

\noindent The Dashnic-Zusmanovich inclusion-exclusion localization set is defined as
\[
\Theta(P(\lambda)) = \bigcap_{i \in {N}} \bigcup_{j \in N\setminus\{i\}} \left( \mathcal{D}_{ij}(P(\lambda)) \setminus {\cal L}_{ij}(P(\lambda)) \right),
\]
\noindent where
\[ 	\mathcal{D}_{ij}(P(\lambda)) = \big\{\mu \in \mathbb{C} : |(P(\mu))_{ii}|\, \big(|(P(\mu))_{jj}| - r_{j}^i(P(\mu))\big) \le r_{i}(P(\mu))\, |(P(\mu))_{ji}|\big\},\]
\[ \begin{aligned}
	{\cal L}_{ij}(P(\lambda)) &= \big\{\mu \in \mathbb{C} : |(P(\mu))_{ii}|\, \big(|(P(\mu))_{jj}| + r_{j}^i(P(\mu))\big)< \\[0.2cm] 
	&\hspace*{5cm} < \big(|(P(\mu))_{ij}| - r_{i}^j(P(\mu))\big)\, |(P(\mu))_{ji}|\big\}.
\end{aligned}
\]
It is known that these sets contain all the eigenvalues of $P(\lambda)$ (see Theorems 10, 33, and Definition 43 in \cite{MPP}). 


\subsection{$S-$SDD inclusion sets}

\noindent For any nonempty $S \subseteq N$, we denote the complement of $S$ as $\overline{S} := N\setminus S$, and we define the sets 
$$r_{i}^S(P(\la))=\sum_{j \in S\setminus \{i\}}|(P(\la))_{ij}|, \qquad r_{i}^{\overline{S}}(P(\la))=\sum_{j \in {\overline{S}}\setminus \{i\}}|(P(\la))_{ij}|.$$
Then, in a similar way as in the Preliminaries, we introduce the sets
\[
	\begin{aligned}
		&\Gamma^{S}_{i}(P(\la)) = \left\{\mu \in \Complex :\, 0 \in \Gamma_i(P(\mu)) \right\}
		= \left\{\mu \in \Complex :\, |(P(\mu))_{ii}| \le r^{S}_i(P(\mu)) \right\} ,\\[2mm]
		&V^{S}_{ij}(P(\lambda)) = \left\{ \mu \in \mathbb{C} : 0 \in V^{S}_{ij}(P(\mu)) \right\} =\\[2mm]
		&= \Big\{ \mu \in \mathbb{C} : \left( |(P(\mu))_{ii}| - r^{S}_{i}(P(\mu)) \right) \left( |(P(\mu))_{jj}| - r^{\overline{S}}_{j}(P(\mu)) \right) \leq r^{\overline{S}}_{i}(P(\mu)) \; r^{S}_{j}(P(\mu)) \Big\},
	\end{aligned} 
	\]

\noindent and 
\begin{equation} \label{eq:CKVpoly}
	C^{S}(P(\lambda)) = \Bigg( \bigcup_{i\in S}
	\Gamma^{S}_{i} (P(\lambda)) \Bigg) \bigcup \Bigg( \bigcup_{i\in S, j\in
		\overline{S}} V^{S}_{ij} (P(\lambda)) \Bigg).
\end{equation}
The CKV inclusion set of  $P(\lambda)$ is defined as the intersection over all possible choices of $S$:  
\[
C(P(\lambda))=\bigcap_{\emptyset \ne S\subseteq N} C^S(P(\lambda)).
\] 

\begin{theorem}
	All (finite and infinite) eigenvalues of the matrix polynomial $P(\lambda)$ lie in the CKV set $C(P(\la))$.
\end{theorem}
\begin{proof}
	Let \( \mu \) be a finite eigenvalue of \( P(\lambda) \). Then \( 0 \in \sigma(P(\mu)) \), and consequently \( 0 \in C(P(\mu)) \), which yields \( \mu \in C(P(\lambda)) \).
	
	If \( \mu = \infty \) is an eigenvalue of \( P(\lambda) \), then \( 0 \in \sigma(\widehat{P}(\mu)) \) and, by the same argument applied to the reverse matrix polynomial, \( 0 \in C(\widehat{P}(\mu)) \). Hence, \( \mu = \infty \) lies in \( C(P(\lambda)) \).
\end{proof}
By the definition of the Dashnic-Zusmanovich set in \cite{MP18}
$$
{\cal D}(P(\lambda))=\left\{\mu \in \mathbb{C} : 0 \in {\cal D}(P(\mu))\right\}=\bigcap_{i \in N} \bigcup_{j \in N\setminus\{i\}}{\cal D}_{ij}(P(\lambda)),
$$
it is easy to see that, \( C(P(\lambda)) \subseteq \mathcal{D}(P(\lambda)) \), for any matrix polynomial $P(\lambda)$. Indeed, a scalar \( \mu \in C \) lies in \( C(P(\lambda)) \) if and only if 0 lies in \( C(P(\mu)) \). Since \( C(P(\mu)) \subseteq {\cal D}(P(\mu))\) (see \cite{CKV}), \( \mu \) lies in $ \mathcal{D}(P(\lambda))$ .

Based on equation (\ref{eq:Wij}) for constant matrices, we may now define the corresponding set $W_{ij}^S$ for a matrix polynomial $P(\lambda)$ as follows:
\[
W_{ij}^S(P(\lambda)) := \left\{ \mu \in \mathbb{C} : 0 \in W_{ij}^S(P(\mu)) \right\}.
\]
Then, for every $\mu \in \mathbb{C}$, Lemma \ref{lemma:1} applied to the constant matrix $P(\mu)$ yields the equality
\[
\left( \bigcup_{i \in S} \Gamma_i^S(P(\mu)) \right)
\bigcup
\Bigg( \bigcup_{i \in S, j \in \overline{S}} V_{ij}^S(P(\mu)) \Bigg)
=
\bigcup_{i \in S, j \in \overline{S}} W_{ij}^S(P(\mu)).
\]
Consequently, for any proper subset $S$ of $N$, the set $C^S(P(\lambda))$, defined in (\ref{eq:CKVpoly}), satisfies
\[
C^S(P(\lambda)) = \Bigg\{ \mu \in \mathbb{C} : 0 \in \bigcup_{i \in S,j \in \overline{S}} W_{ij}^S(P(\mu)) \Bigg\}
= \bigcup_{i \in S,j \in \overline{S}} W_{ij}^S(P(\lambda)).
\]
Therefore, the CKV inclusion set for matrix polynomials can be equivalently expressed as
\[
C(P(\lambda)) = \Gamma(P(\lambda)) \bigcap \Bigg(\bigcap_{\emptyset \ne S \subsetneq N} \bigcup_{i \in S , j \in \overline{S}} W_{ij}^S(P(\lambda))\Bigg).
\]


\subsection{$S-$SDD inclusion-exclusion sets}

Next, we extend the notion of exclusion sets from the constant matrix case (introduced in Section~\ref{sec:SSDDinclexcl}) to matrix polynomials, in an attempt to improve the CKV inclusion set. To this end, we give the following definitions.

\noindent For any $i \in S$ and $j \in \overline{S}$, we define the exclusion set
\[
\Psi_{ij}^S(P(\lambda)) :=
\left\{
\mu \in \mathbb{C} : 0 \in \Psi_{ij}^S(P(\mu))
\right\},
\]
where $\Psi_{ij}^S(P(\mu))$ is given by the constant-matrix definition
in (\ref{eq:psi}).
Equivalently, we may write explicitly:
\[
\begin{aligned}
	&	\Psi_{ij}^S(P(\lambda)) =
	\Bigg\{
	\mu \in \mathbb{C} :\,
	\Big( |(P(\mu))_{ii}| + r_{i}^S(P(\mu)) \Big)
	\Big( |(P(\mu))_{jj}| + r_{j}^{\overline{S}}(P(\mu)) \Big)< \\
	&<
	\Big( \max\big\{0,\, 2|(P(\mu))_{ij}| - r_{i}^{\overline{S}}(P(\mu)) \big\} \Big)
	\Big( \max\big\{0,\, 2|(P(\mu))_{ji}| - r_{j}^S(P(\mu)) \big\} \Big)
	\Bigg\}.
\end{aligned}
\]
For any nonempty $S \subseteq N$, we define the corresponding inclusion-exclusion set for the matrix polynomial $P(\lambda)$ as
\[
\Xi^S(P(\lambda)) :=
\left\{
\begin{array}{lll}
	\displaystyle \bigcup_{i \in S , j \in \overline{S}}
	\Big( W_{ij}^S(P(\lambda)) \setminus \Psi_{ij}^S(P(\lambda)) \Big) & \text{if} & \emptyset \neq S \subsetneq N ,\\[0.5cm]
	\Omega(P(\lambda)) & \text{if} & S = N.
\end{array}
\right.
\]
Then the CKV inclusion-exclusion set for the matrix polynomial $P(\lambda)$ is obtained from the intersection over all possible choices of S:
\[
\Xi(P(\lambda)) := \bigcap_{\emptyset \neq S \subseteq N} \Xi^S(P(\lambda)).
\]
\begin{theorem}
	All (finite and infinite) eigenvalues of $P(\lambda)$ lie in the CKV inclusion-exclusion set $\Xi(P(\lambda))$.
\end{theorem}
\begin{proof}
	We first consider a finite eigenvalue $\mu \in \sigma(P(\lambda))$. By definition, $0 \in \sigma(P(\mu))$, i.e. $0$ is an eigenvalue of the constant matrix $P(\mu)$.
	
	For any nonempty subset $S \subseteq N$, we apply the corresponding constant-matrix result:
	\begin{itemize}
		\item If $S = N$, then by Corollary~\ref{posl} applied to the constant matrix $P(\mu)$, we have
		\[0 \in \Xi^{N}(P(\mu)) = \Omega(P(\mu)).\]
		\item If $\emptyset \neq S \subsetneq N$, then by Theorem~\ref{glavna} applied to the constant matrix $P(\mu)$, there exist indices $i \in S$ and $j \in \overline{S}$ such that
		\[0 \in W_{ij}^S(P(\mu)) \setminus \Psi_{ij}^S(P(\mu)).\]
	\end{itemize}
	
	\noindent	In either case, for every $S \subseteq N$ with $S \neq \emptyset$, we have $0 \in \Xi^S(P(\mu))$. Hence, ${\displaystyle 0 \in \bigcap_{\emptyset \neq S \subseteq N} \Xi^S(P(\mu)) = \Xi(P(\mu)).}$
	
	\noindent	Now, by the definition of the polynomial sets, for any $\mu \in \mathbb{C}$, we have the equivalence
	\[
	\mu \in \Xi^S(P(\lambda)) \quad \Longleftrightarrow \quad 0 \in \Xi^S(P(\mu)),
	\]
	and similarly
	\[
	\mu \in \Xi(P(\lambda)) \quad \Longleftrightarrow \quad 0 \in \Xi(P(\mu)).
	\]
	Therefore, since $0 \in \Xi(P(\mu))$, it follows that $\mu \in \Xi(P(\lambda))$.
	
	Finally, if $\mu = \infty$ is an eigenvalue of $P(\lambda)$, then $0 \in \sigma(\widehat{P}(\mu))$, where $\widehat{P}(\lambda)$ is the reverse matrix polynomial. Applying the same argument to $\widehat{P}(\lambda)$ yields $\infty \in \Xi(P(\lambda))$.
	
	Thus, all eigenvalues of $P(\lambda)$ lie in $\Xi(P(\lambda))$.
\end{proof}

Before discussing some geometric properties of the proposed CKV inclusion-exclusion set, we present an illustrative example to compare it with the existing inclusion-exclusion sets. In this example, the new set is tighter than both the Ger\v sgorin inclusion-exclusion set and the Dashnic–Zusmanovich inclusion-exclusion set, and offers a noticeable improvement over the standard CKV inclusion set.

\begin{example} \em \label{P17}
	For the $5 \times 5$ matrix polynomial
	\[
	P(\lambda) = 
	\left[
	\begin{array}{ccccc}
		\lambda^2+11 & 3 & -1 & 2 & 1 \\
		1 & \lambda^2+10 & 0 & 1 & -3 \\
		0 & 0 & \lambda^2-1 & -7 & 1 \\
		-1 & 0 & 7 & \lambda^2-7 & 0 \\
		1 & 1 & 0 & -3 & \lambda-12
	\end{array}
	\right]
	\]
	the following figures show all above-mentioned localization sets.
\begin{figure}[!ht]
	\centering
	
	\begin{minipage}[t]{0.38\textwidth}
		\centering
		\includegraphics[width=\linewidth]{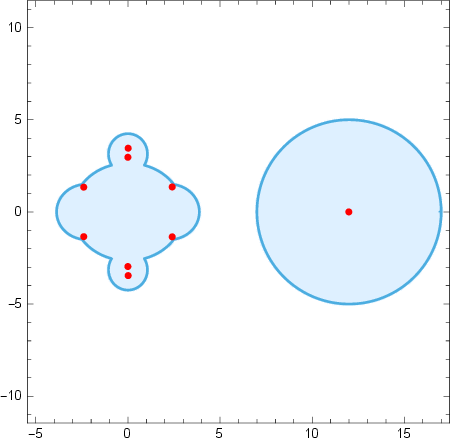}
		\captionof{figure}{Ger\v sgorin set of $P(\lambda)$.}
		\label{fig:p17_Gerg}
	\end{minipage}
\hspace{0.03\textwidth}
	\begin{minipage}[t]{0.38\textwidth}
		\centering
		\includegraphics[width=\linewidth]{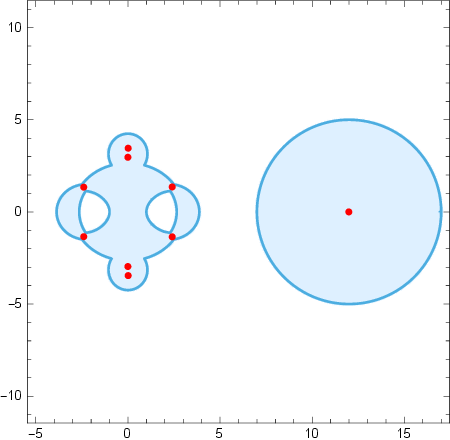}
		\captionof{figure}{Ger\v sgorin incl-excl set of $P(\lambda)$.}
		\label{fig:p17_GergExcl}
	\end{minipage}
	
	\vspace{0.1em}
	
	\begin{minipage}[t]{0.38\textwidth}
		\centering
		\includegraphics[width=\linewidth]{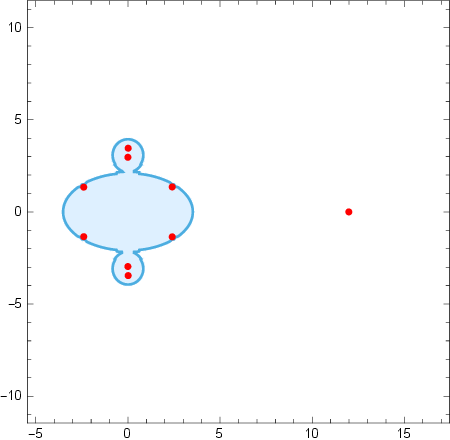}
		\captionof{figure}{DZ set of $P(\lambda)$.}
		\label{fig:p17_dz}
	\end{minipage}
\hspace{0.03\textwidth}
	\begin{minipage}[t]{0.38\textwidth}
		\centering
		\includegraphics[width=\linewidth]{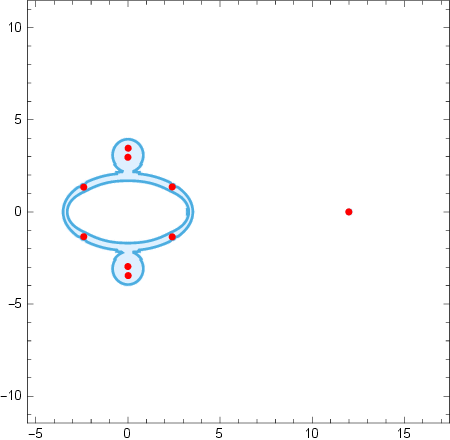}
		\captionof{figure}{DZ incl-excl set of $P(\lambda)$.}
		\label{fig:p17_dzexcl}
	\end{minipage}
	
	\vspace{0.1em}
	
	\begin{minipage}[t]{0.38\textwidth}
		\centering
		\includegraphics[width=\linewidth]{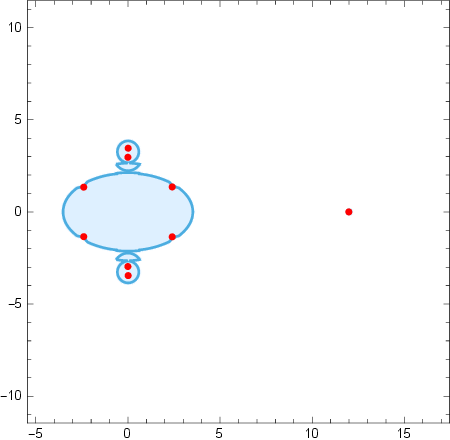}
		\captionof{figure}{CKV set of $P(\lambda)$.}
		\label{fig:p17_ckv}
	\end{minipage}
\hspace{0.03\textwidth}
	\begin{minipage}[t]{0.38\textwidth}
		\centering
		\includegraphics[width=\linewidth]{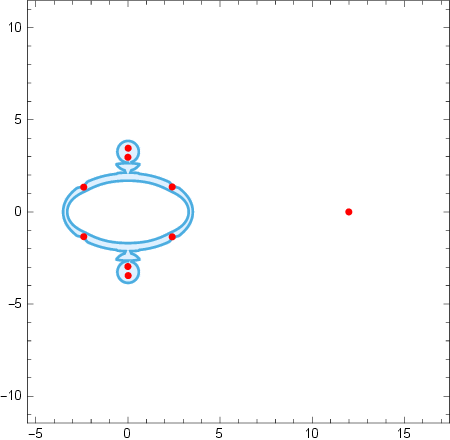}
		\captionof{figure}{CKV incl-excl set of $P(\lambda)$.}
		\label{fig:p17_ckvexcl}
	\end{minipage}
	
\end{figure}
\end{example}

\subsection{Some geometric properties of the CKV inclusion-exclusion set}

\begin{proposition}\label{prop:Psi_unbounded}
	Let $P(\lambda)$ be an $n \times n$ matrix polynomial 
	of degree m, and let $S$ be a proper nonempty subset of $N$.
	Suppose that there exists a pair $(i,j) \in S \times \overline{S}$ (necessarily unique) satisfying
	\begin{equation}\label{eq:exclusion_condition_Am}
		2|(A_m)_{ij}| > r_{i}^{\overline{S}}(A_m) 
		\quad\text{and}\quad 
		2|(A_m)_{ji}| > r_{j}^{S}(A_m).
	\end{equation}
	\begin{enumerate}[label=(\roman*)]
		\item If \, $0 \in \Psi_{ij}^S(A_m)$, then the set \(\Psi_{ij}^S(P(\lambda))\) is unbounded.
		\item If \, \(|(A_m)_{ii}| + r_{i}^{S}(A_m) = 0 \;\; \text{or} \;\;
		|(A_m)_{jj}| + r_{j}^{\overline S}(A_m) = 0\), then \(0\in \Psi_{ij}^S(A_m)\) and the set \(\Psi_{ij}^S(P(\lambda))\) is unbounded.
	\end{enumerate}
\end{proposition}

\begin{proof}
	
	First, observe that if condition \eqref{eq:exclusion_condition_Am} fails, $\Psi_{ij}^S(A_m)=\emptyset$.
	
	Conversely, when \eqref{eq:exclusion_condition_Am} holds, both max terms are strictly positive and simplify to:
	\[
	\max\big\{0, 2|(A_m)_{ij}| - r_{i}^{\overline{S}}(A_m)\big\} = 2|(A_m)_{ij}| - r_{i}^{\overline{S}}(A_m) > 0,
	\]
	\[
	\max\big\{0, 2|(A_m)_{ji}| - r_{j}^{S}(A_m)\big\} = 2|(A_m)_{ji}| - r_{j}^{S}(A_m) > 0.
	\]
	Now recall that \(\widehat{P}(\lambda) = A_0\lambda^m + \cdots + A_m \) and if \(\lambda \neq 0\), \(\widehat{P}(\lambda) = \lambda^m P(\frac{1}{\lambda})\). For \(\lambda = 0\), it follows that \(\widehat{P}(0)=A_m\).
	
	\begin{enumerate}[label=(\roman*)]
		\item	Since \(0 \in \Psi_{ij}^S(A_m)\), \(0 \in \Psi_{ij}^S(\widehat{P}(0))\), that is,
		\[
		\begin{aligned}
			&\left( |(\widehat{P}(0))_{ii}| + r_{i}^S(\widehat{P}(0)) \right)
			\left( |(\widehat{P}(0))_{jj}| + r_{j}^{\overline{S}}(\widehat{P}(0)) \right)<\\
			&<
			\left( 2|(\widehat{P}(0))_{ij}| - r_{i}^{\overline{S}}(\widehat{P}(0)) \right)
			\left( 2|(\widehat{P}(0))_{ji}| - r_{j}^S(\widehat{P}(0)) \right).
		\end{aligned}\]
				By continuity, there exists a real number \(r > 0\) such that for every \(\widehat{\mu} \in \mathbb{C}\) with \(|\widehat{\mu}| \le r\),
		\[
		\begin{aligned}
			&	\left( |(\widehat{P}(\widehat{\mu}))_{ii}| + r_{i}^S(\widehat{P}(\widehat{\mu})) \right)
			\left( |(\widehat{P}(\widehat{\mu}))_{jj}| + r_{j}^{\overline{S}}(\widehat{P}(\widehat{\mu})) \right)<\\ & < \left( 2|(\widehat{P}(\widehat{\mu}))_{ij}| - r_{i}^{\overline{S}}(\widehat{P}(\widehat{\mu})) \right)
			\left( 2|(\widehat{P}(\widehat{\mu}))_{ji}| - r_{j}^S(\widehat{P}(\widehat{\mu})) \right).
		\end{aligned}\]	
		Taking \(\mu = 1/\widehat{\mu}\), we have \(|\mu| \ge r^{-1}\) and \(\widehat{P}(\widehat{\mu}) = \widehat{\mu}^m P(\frac{1}{\widehat{\mu}})\). Substituting \(\widehat{P}(\widehat{\mu}) = \widehat{\mu}^m P(\mu)\) into the inequality and cancelling the common factor \(|\widehat{\mu}|^{2m} > 0\), we obtain
		\[
		\begin{aligned}
			&\left( |(P(\mu))_{ii}| + r_{i}^S(P(\mu)) \right)
			\left( |(P(\mu))_{jj}| + r_{j}^{\overline{S}}(P(\mu)) \right)<\\  &< \left( 2|(P(\mu))_{ij}| - r_{i}^{\overline{S}}(P(\mu)) \right)
			\left( 2|(P(\mu))_{ji}| - r_{j}^S(P(\mu)) \right).
		\end{aligned}\]
		Consequently, \(\{ \mu \in \mathbb{C} : |\mu| \geq r^{-1} \} \subseteq \Psi_{ij}^S( P(\lambda))\). Therefore, the exclusion set \(\Psi_{ij}^S( P(\lambda))\) is unbounded.\\
		
		\item
		Assume without loss of generality that
		\[
		|(A_m)_{ii}| + r_{i}^{S}(A_m) = 0.
		\]
		Then the left-hand side of the definition of \(0\in \Psi_{ij}^S(A_m)\) is
		\[
		\left(|(A_m)_{ii}| + r_{i}^{S}(A_m)\right)
		\left(|(A_m)_{jj}| + r_{j}^{\overline S}(A_m)\right) = 0.
		\]
		By assumption \eqref{eq:exclusion_condition_Am}, the right-hand side is strictly positive. Hence,
		\[
		0 < \left(2|(A_m)_{ij}| - r_{i}^{\overline S}(A_m)\right)
		\left(2|(A_m)_{ji}| - r_{j}^S(A_m)\right),
		\]
		which is precisely the condition that \(0\in \Psi_{ij}^S(A_m)\).
		
	\noindent The unboundedness of \(\Psi_{ij}^S(P(\lambda))\) then follows immediately from part (i).

	\end{enumerate}
\end{proof}

\begin{proposition}\label{prop:PsiUnbounded2}
	If the set $\Psi_{ij}^S(P(\lambda))$ is unbounded for some $i\in S$, $j\in \overline{S}$,
	then $0$ lies in the closure of $\Psi_{ij}^S(A_m)$.
\end{proposition}

\begin{proof}
	Assume that $\Psi_{ij}^S(P(\lambda))$ is unbounded. Then there exists a sequence
	$\{\mu_\ell\}_{\ell=1}^{\infty}$ in $\Psi_{ij}^S(P(\lambda))$ such that $|\mu_\ell|\to\infty$.
	For every $\ell$, by the definition of $\Psi_{ij}^S$, we have	
	\begin{equation}\label{eq:1Psiinequality}
		\begin{aligned}
			&\left( |(P(\mu_\ell))_{ii}| + r_{i}^{S}(P(\mu_\ell)) \right)
			\left( |(P(\mu_\ell))_{jj}| + r_{j}^{\overline{S}}(P(\mu_\ell)) \right)<\\ &< \left( 2|(P(\mu_\ell))_{ij}| - r_{i}^{\overline{S}}(P(\mu_\ell)) \right)
			\left( 2|(P(\mu_\ell))_{ji}| - r_{j}^{S}(P(\mu_\ell)) \right).
		\end{aligned}
	\end{equation}
	Since $\displaystyle P(\mu_\ell)=\sum_{k=0}^m A_k \mu_\ell^k$, we can write $\displaystyle
	(P(\mu_\ell))_{ij} = \mu_\ell^m \sum_{k=0}^m (A_k)_{ij} \mu_\ell^{k-m}.$ Similarly,
	\(\displaystyle
	r_{i}^T(P(\mu_\ell)) = |\mu_\ell|^m \, r_{i}^T\!\left(\sum_{k=0}^m A_k \mu_\ell^{k-m}\right),
	\)
	for any subset $T\subseteq N$.
	
	Dividing both sides of \eqref{eq:1Psiinequality} by $|\mu_\ell|^{2m}>0$, we obtain the equivalent inequality
 \begin{equation}\label{eq:2}
			\begin{aligned}
				&\Bigg( \Big|\sum_{k=0}^m (A_k)_{ii} \mu_\ell^{k-m}\Big| + r_{i}^{S}\!\Big(\sum_{k=0}^m A_k \mu_\ell^{k-m}\Big) \Bigg)
				\Bigg( \Big|\sum_{k=0}^m (A_k)_{jj} \mu_\ell^{k-m}\Big| + r_{j}^{\overline{S}}\!\Big(\sum_{k=0}^m A_k \mu_\ell^{k-m}\Big) \Bigg)< \\
				&< \Bigg( 2\Big|\sum_{k=0}^m (A_k)_{ij} \mu_\ell^{k-m}\Big| - r_{i}^{\overline{S}}\!\Big(\sum_{k=0}^m A_k \mu_\ell^{k-m}\Big) \Bigg)
				\Bigg( 2\Big|\sum_{k=0}^m (A_k)_{ji} \mu_\ell^{k-m}\Big| - r_{j}^{S}\!\Big(\sum_{k=0}^m A_k \mu_\ell^{k-m}\Big) \Bigg).
			\end{aligned}
	\end{equation}
	\noindent	Passing to the limit as $\ell\to\infty$ in \eqref{eq:2}, and using the fact that $\mu_\ell^{k-m}\to 0$ for every $k<m$ while $\mu_\ell^{m-m}=1$, we get	
	\begin{equation}\label{eq:3}
		\begin{aligned}
			&\left( |(A_m)_{ii}| + r_{i}^{S}(A_m) \right)
			\left( |(A_m)_{jj}| + r_{j}^{\overline{S}}(A_m) \right)\le \\ & \le \left( 2|(A_m)_{ij}| - r_{i}^{\overline{S}}(A_m) \right)
			\left( 2|(A_m)_{ji}| - r_{j}^{S}(A_m) \right),
		\end{aligned}
	\end{equation}
	where the strict inequality weakens to non-strict in the limit.
	Inequality \eqref{eq:3} is precisely the condition that $0$ belongs to the closure of $\Psi_{ij}^S(A_m)$. Hence, the proof is complete.
\end{proof}


\begin{example} \em \label{P18P19}
	
	Consider the $3 \times 3$ matrix polynomial
	\[
	P(\lambda) = 
	\left[
	\begin{array}{ccc}
		2 - 2i\lambda & (1+2i) + i\lambda + 4i\lambda^2 & 2 + \lambda + 2i\lambda^2 \\
		(1+i) + 5i\lambda - 12i\lambda^2 & 4i + 3i\lambda - 6i\lambda^2 & -5i - 4\lambda + 19i\lambda^2 \\
		(1-i) + i\lambda + 9i\lambda^2 & -i\lambda + 2i\lambda^2 & 2i + 6i\lambda^2
	\end{array}
	\right]
	\]
	
	Figures \ref{fig:p18_ckv} and \ref{fig:p18_ckvexcl} illustrate the unbounded CKV set of $P(\lambda)$ and the improvement achieved by the corresponding bounded CKV inclusion--exclusion set, whereas in Figures \ref{fig:p18_dz} and \ref{fig:p18_dzexcl} the DZ and DZ inclusion--exclusion sets remain unbounded.
	
\begin{figure}[!htbp]
	\centering
	
	\begin{minipage}[t]{0.43\textwidth}
		\centering
		\includegraphics[width=\linewidth]{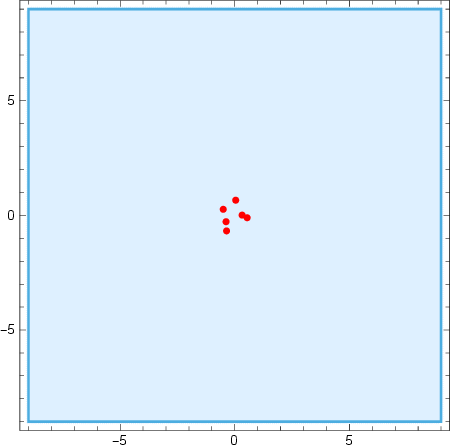}
		\captionof{figure}{CKV set of $P(\lambda)$.}
		\label{fig:p18_ckv}
	\end{minipage}
	\hspace{0.03\textwidth}
	\begin{minipage}[t]{0.43\textwidth}
		\centering
		\includegraphics[width=\linewidth]{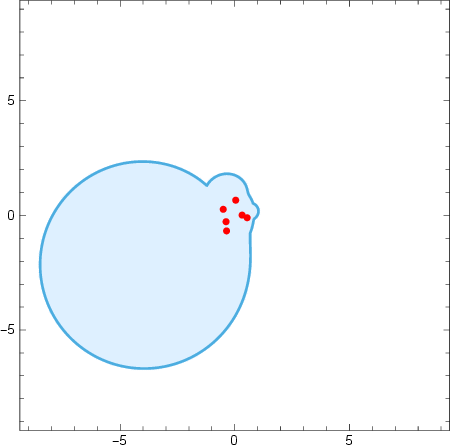}
		\captionof{figure}{CKV incl-excl set of $P(\lambda)$.}
		\label{fig:p18_ckvexcl}
	\end{minipage}
	
	\vspace{0.3em}
	
	\begin{minipage}[t]{0.43\textwidth}
		\centering
		\includegraphics[width=\linewidth]{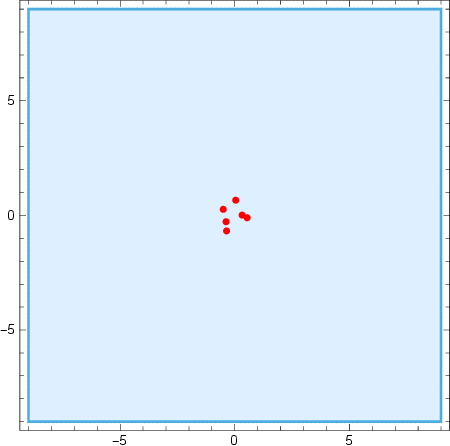}
		\captionof{figure}{DZ set of $P(\lambda)$.}
		\label{fig:p18_dz}
	\end{minipage}
	\hspace{0.03\textwidth}
	\begin{minipage}[t]{0.43\textwidth}
		\centering
		\includegraphics[width=\linewidth]{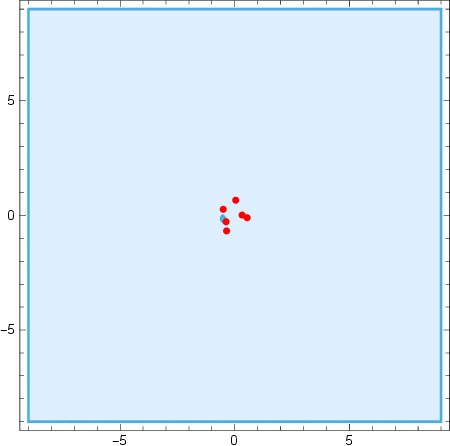}
		\captionof{figure}{DZ incl-excl set of $P(\lambda)$.}
		\label{fig:p18_dzexcl}
	\end{minipage}
	
\end{figure}
	
	Furthermore, for \(S=\{2,3\}\), \(i=2,\ j=1\), Figures~\ref{fig1},~\ref{fig2}, and~\ref{fig3} illustrate the sets \(W_{ij}^S(P(\lambda))\), \(\Psi_{ij}^S(P(\lambda))\), and \(\Xi_{ij}^S(P(\lambda))\), respectively, where \(W_{ij}^S(P(\lambda))\) and \(\Psi_{ij}^S(P(\lambda))\) are unbounded (as \(0\in W_{ij}^S(A_2)\) and \(0\in \Psi_{ij}^S(A_2)\)), while \(\Xi_{ij}^S(P(\lambda))\) is bounded because \(0\notin \Xi_{ij}^S(A_2)\), verifying the assertion of Proposition~\ref{prop:Psi_unbounded}.
	
	
\begin{figure}[!htbp]
	\centering
	
	\begin{minipage}[t]{0.30\textwidth}
		\centering
		\includegraphics[width=\linewidth]{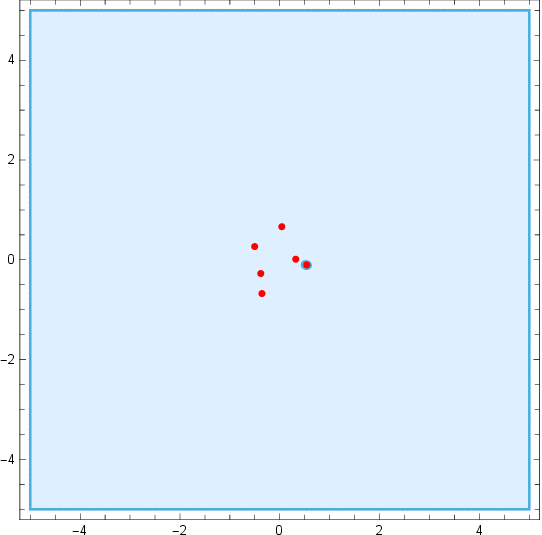}
		\captionof{figure}{$W_{21}^{\{2,3\}}(P(\lambda))$.}
		\label{fig1}
	\end{minipage}
	\hspace{0.02\textwidth}
	\begin{minipage}[t]{0.30\textwidth}
		\centering
		\includegraphics[width=\linewidth]{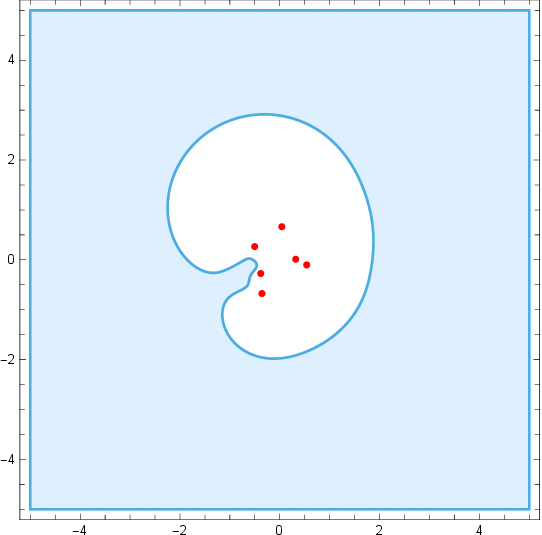}
		\captionof{figure}{$\Psi_{21}^{\{2,3\}}(P(\lambda))$.}
		\label{fig2}
	\end{minipage}
	\hspace{0.02\textwidth}
	\begin{minipage}[t]{0.30\textwidth}
		\centering
		\includegraphics[width=\linewidth]{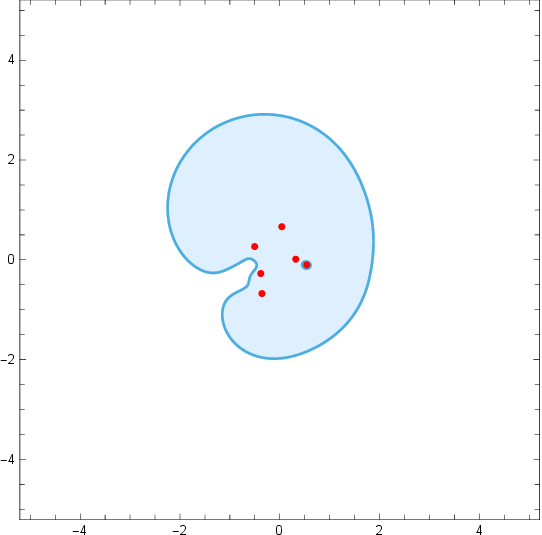}
		\captionof{figure}{$\Xi_{21}^{\{2,3\}}(P(\lambda))$.}
		\label{fig3}
	\end{minipage}
	
\end{figure}
	
\end{example}
\section{Conclusions}

In this paper, we introduced a new family of inclusion-exclusion eigenvalue localization sets based on the framework of $S-$strict diagonal dominance. By combining the ideas of CKV localization with exclusion regions, we obtained a unified construction that extends the classical inclusion-exclusion approach to arbitrary nonempty subsets of the index set.

The proposed framework naturally contains both the Ger\v sgorin and the Dashnic-Zusmanovich inclusion-exclusion localization sets as special cases. By intersecting the corresponding $S-$SDD inclusion-exclusion regions over all nonempty subsets $S$, we derived the CKV inclusion-exclusion localization set, which is, in general, sharper than the previously known localization sets. 

The theory was further extended to matrix polynomials by introducing the corresponding $S-$SDD inclusion and inclusion-exclusion localization sets. We proved that all finite and infinite eigenvalues of a matrix polynomial belong to the proposed localization regions and established several structural properties of the associated exclusion sets.

Finally, the presented numerical examples demonstrate that the proposed approach can substantially reduce the size of the localization region compared with existing ones. We believe that the flexibility of the $S-$SDD framework makes it a promising basis for further developments, including refined localization techniques, computational algorithms, and applications to structured matrices and polynomial eigenvalue problems.


\section*{Acknowledgements}
The work of the third author has been supported by the Ministry of Science, Technological Development and Innovation (Contract No. $451-03-34/2026-03/200156$) and the Faculty of Technical Sciences, University of Novi Sad through project ``Scientific and Artistic Research Work of Researchers in Teaching and Associate Positions at the Faculty of Technical Sciences, University of Novi Sad 2026'' (No. $01-3609/1$).


\end{document}